\documentclass[11pt]{article}
\usepackage{geometry}                
\usepackage{amsthm,amsmath,amssymb}

\usepackage{graphicx}

\usepackage[colorlinks=true,citecolor=black,linkcolor=black,urlcolor=blue]{hyperref}

\usepackage{pgfplots}
\pgfplotsset{width=6cm,compat=1.8}
\usepackage{mathtools}
\usepackage{color}
\usepackage{tikz}
\usetikzlibrary{arrows}
\usepackage{subfigure}

\theoremstyle{plain}
\newtheorem{theorem}{Theorem}
\newtheorem{lemma}[theorem]{Lemma}
\newtheorem{corollary}[theorem]{Corollary}

\theoremstyle{definition}
\newtheorem{definition}[theorem]{Definition}
\newtheorem{example}[theorem]{Example}
\newtheorem{conjecture}[theorem]{Conjecture}

\theoremstyle{remark}
\newtheorem{remark}[theorem]{Remark}

\newcommand{\CA}{{\mathcal{A}}}

\newcommand{\Cl}{{\textit{l}}}
\newcommand{\sign}{\textsf{sign}}

\newcommand{\bae}{\begin{equation}\begin{aligned}}
\newcommand{\eae}{\end{aligned}\end{equation}}

\newcommand{\pr}{\mathbb{P}}
\newcommand{\Z}{\mathbb{Z}}
\newcommand{\T}{\mathbb{T}}
\newcommand{\N}{\mathbb{N}}

\newcommand{\aseq}{{\mathbf{a}}}
\newcommand{\bseq}{{\mathbf{b}}}
\newcommand{\cseq}{{\mathbf{c}}}

\newcommand{\period}[1]{\overline{(#1)}}

\newcommand{\pere}[1]{{#1}_0}

\newcommand{\fils}[2]{{#1}_{#2}}

\title{Infinite excursions of rotor walks on regular trees}
\author{Sebastian M\"uller\\
\small Aix Marseille Univ\\[-0.8ex]
\small CNRS, Centrale Marseille\\[-0.8ex]
\small I2M\\[-0.8ex]
\small Marseille, France\\
\small\tt sebastian.muller@univ-amu.fr\\
\and
 Tal Orenshtein\thanks{The work of T.O.\ was supported by the Labex Milyon (ANR-10-LABX-0070) of Universit\'e de Lyon, within the program "Investissements d'Avenir" (ANR-11-IDEX-0007) operated by the French National Research Agency (ANR).
}\\
\small Humboldt Universit\"at zu Berlin\\[-0.8ex]
\small Institut f\"ur Mathematik\\[-0.8ex]
\small Berlin, Germany\\
\small and \\
\small Technische Universit\"at Berlin\\[-0.8ex]
\small Institut f\"ur Mathematik\\[-0.8ex]
\small Berlin, Germany\\
\small\tt orenshtein@tu-berlin.de
}

\date{}
                                    
\begin{document}

\maketitle   

\begin{abstract}
A rotor configuration on a graph contains in every vertex an infinite ordered sequence of rotors, each is pointing to a neighbor of the vertex. After sampling a configuration according to some probability measure, 
a rotor walk is a deterministic process:
at each step it chooses the next unused rotor in its current location, and uses it to jump to the neighboring vertex
to which it points. Rotor walks capture many aspects of the expected behavior of simple random walks.
However, this similarity breaks down for the property of having an infinite excursion.
In this paper we study that question for natural random configuration models on regular trees. Our results suggest that in this context the rotor model behaves like the simple random walk unless it is not ``close to''
the standard rotor-router model.

  

\end{abstract}
\section{Introduction}

\subsection{Informal motivation}
We consider first rotor walks on $\N$: on each vertex $n$ there is an infinite \emph{rotor sequence} $\aseq_{n}\in \{-1,1\}^{\N}$ pointing to one of the neighbors.
The walk starts in the origin. Inductively, the walk being at $n$ follows the direction of the first rotor in $\aseq_{n}$ and deletes this rotor.
Assume that each $\aseq_n$ is non-degenerate, i.e. it contains infinitely many $-1$'s and $+1$'s.  The first question that we ask is whether having at each vertex a ``local drift'' zero implies ``recurrence'' of the rotor walk;
call a rotor sequence $\aseq$ \emph{$L$-periodic} if $\aseq(x)=\aseq(x+L)$ for all $x\in \N$ and \emph{balanced} if there are as many $-1$'s as $+1$'s per period.
\begin{itemize}
\item Fix a period $L$. In optimizing over all $L$-periodic balanced choices of $\aseq_{n}$, $n\in\N$, what is the maximal number of ``infinite excursions'' that can be achieved?
\item Choose $\aseq_{n}$, $n\in \N,$ in an i.i.d.\ way. What are the conditions which ensure that the rotor walk is recurrent a.s.?
\end{itemize}
Consider the same model on the binary tree $\T_{2}$. Now, we have three possible directions to choose from. Let $0$ describe the direction  towards the root and let $1$ and $2$ stand for pointing towards the two children respectively, see Figure \ref{fig:ex rr on T2}.
A rotor sequence $\aseq$ now takes values in $\{0,1,2\}^{\N}$.
It is well known that the simple random walk on the binary tree is transient. However, the following rotor(-router) 
walk is recurrent, see \cite{Pa:07, AnHo:11}. Consider the $3$-periodic rotor sequences
\begin{equation*}
\aseq^{(1)}=(0,1,2,0,1,2,\ldots),  \aseq^{(2)}=(1,2,0,1,2,0,\ldots) \mbox{ and }\aseq^{(3)}=(2,0,1,2,0,1,\ldots),
\end{equation*}
and choose for each vertex of $\T_{2}$ independently one of these three sequences with equal probability. 
This behavior difference is somewhat surprising, since rotor walks share many properties with the simple random walk.
\begin{itemize}
\item Is it a general phenomenon that (periodic and balanced) rotor walks are recurrent on $\T_{2}$?
\end{itemize}
We answer this question negatively. In particular, consider the $6$-periodic sequences
\begin{equation*}
\aseq^{(1)}=(0,0,1,2,1,2,\ldots),  \aseq^{(2)}=(1,1,2,0,2,0,\ldots) \mbox{ and }\aseq^{(3)}=(2,2,0,1,0,1,\ldots),
\end{equation*}
then the corresponding rotor walk in the i.i.d.\  uniform configuration model is transient a.s.
\begin{itemize}
\item Which sequences $\aseq$ give rise to recurrent rotor walks?
\end{itemize}
\begin{itemize}
\item Are there interesting i.i.d.\ rotor configurations on $\T_{d}, d\geq 3,$ that are recurrent?
\end{itemize}

\subsection{General introduction and results}
A rotor walk on a graph is a deterministic process where a particle is routed through the vertices of a graph. At each vertex the particle is
routed to one of the neighboring vertices following a prescribed periodic sequence, called  the \textit{rotor sequence}.
In the classical model, called \textit{rotor-router walk}, the rotor sequence is a fixed cyclic order of the neighboring vertices. 
For an overview of this model and its classic properties we refer to the expository paper \cite{HoLietal:08}.
In this paper we consider configurations that may arise from any non-degenerate sequence. By \textit{non-degenerate} we mean
that there are infinitely many rotors pointing to every neighboring vertex. In this more general case we speak of \textit{rotor walks}. Note that this model was also introduced in \cite{Wi:96} as \textit{stack walks}.

Rotor walks capture in many aspects the expected behavior of simple random walks, but with significantly reduced fluctuations compared to a typical random walk trajectory; for more details see \cite{CoSp:06, FrLe:11, HoPr:10, Kl:05}.
However, this similarity breaks down when one looks at the property of being recurrent or transient. In fact, the rotor walk may
behave differently than the corresponding random walk. We say that a rotor walk which started at the origin with initial rotor configuration $\rho$ is \textit{recurrent} if it returns to the origin infinitely many times; in this case we say that the rotor configuration $\rho$ is recurrent. Otherwise we say that the rotor walk is \textit{transient} or that  the rotor configuration $\rho$ is transient.
In the recurrent regime all excursions (from the origin) are finite. However, in the transient regime there is a first \textit{infinite excursion}. This excursion eventually leaves every finite ball around the origin and hence when it reaches ``infinity'' it  leaves a well defined rotor configuration. For this reason we can start a new walk after the first infinite excursion  in the origin and proceed inductively.

In \cite{AnHo:12} existence of recurrent initial configurations for the  rotor-router walk on many graphs, including $\Z^{d}$, and planar graphs with locally finite embeddings is proved.
An example of an initial rotor-router configuration on $\Z^{2}$ for which the rotor-router walk is recurrent was given earlier in \cite[Theorem 5]{HoPr:10}. See also \cite{FletAl:14} for initial rotor-router configurations with all rotors aligned on $\mathbb{Z}^d$. Infinite excursions of rotor-router walks on homogeneous trees were studied in \cite{Pa:07} and \cite{LaLe:09}. On general trees, the issue of transience and recurrence of rotor-router walks was studied in detail in \cite{AnHo:11}. An extension for random initial configuration of rotor-router walks was made in \cite{HuSa:12}
on directed covers of graphs (periodic trees) and in \cite{HuMuSa:15} on Galton-Watson trees.

In this paper we give criteria for recurrence and transience of rotor walks on $d$-ary trees $\T_{d}, d\geq 1$. Recurrence of rotor-router walks on $\Z$ (and similarly on $\N=\T_{1}$) is rather obvious.
Indeed, if we start a rotor-router walk on $\Z$ with i.i.d.\ uniform initial configuration then the walk has a simple structure. It follows the rotors in one direction until it
hits a rotor pointing in the opposite direction. The walk reverses its direction and retraces its path entirely and continues until it hits a rotor pointing in the opposite direction and so on.
This behavior was presented in  \cite{PretAl:96} as an example of self-organization.

A fundamental tool used in this paper is a connection, observed in \cite{harris1952first},  between nearest neighbor walks and Galton-Watson processes.  In the context of random walks in random environments, 
its usefulness was demonstrated in the well-known paper \cite{kesten1975limit}. 
In the special case of the rotor-router model a more specific construction of a Galton-Watson process was used to prove transience criteria
in \cite{Pa:07} and \cite{AnHo:11}.

Rotor walks can be seen as a special case of excited random walks
by regarding the rotors as non-elliptic cookies. This deterministic point of view is already observed and extensively 
used in \cite{ABO} and \cite{amir2016excited}, where rotors are called `arrows'. In the context of excited random walks on the one-dimensional lattice the relation to   Galton-Watson processes was used first in \cite{basdevant2008speed}. 
The case of excited random walks on regular trees and their relations to survival of multi-type Galton-Watson processes  was introduced by the same authors in \cite{BS2009recurrence}. The interested reader may find  more details on excited random walks in \cite{benjamini2003excited} or in the survey \cite{kosygina2012excited}.

In our setting the relevant process is based on \cite{BS2009recurrence} and can be considered as a multi-type Galton-Watson process with \emph{a priori} infinitely many types, 
see Chapter \ref{sec:d-arry}. For general background on multi-type Galton-Watson processes we refer to \cite{AN:72, Ha:63}.

In the first part of paper we give criteria for recurrence for rotor walks on $\N$, see Theorem \ref{theorem:rec trans Z}, Corollary \ref{cor:random_starting_point_is_recurrent} and Theorem \ref{thm:random_starting_point_non-blanced_is_transient}.
Moreover, we observe another phenomenon of self-organization; we consider the case where a rotor sequence at some vertex can be any $L$-periodic (i.e. has period $L$)
and balanced sequence (i.e., there are as many rotors pointing to the right as pointing to the left in each period). In this case there are at most $L/2$ infinite excursions, see Theorem \ref{theorem:bounded infinite excursion N}. In other words, while there might be infinite excursions for the first walks the system organizes itself in such a way that after at most $L/2$ infinite excursions
it behaves as it ``should'', namely it is recurrent.

In the second part we consider rotor configurations on the homogeneous tree $\T_{d}, d\geq 2$. We give a criterion for recurrence and transience for the general model, see Theorem \ref{theorem:bal rc dary}. This criterion is based on the fact that
transience of the rotor walk is equivalent to survival of a multi-type Galton-Watson process.
We then show that $\T_{2}$ can be considered as the critical case in the following sense. We choose uniformly a rotation of a fixed rotor sequence 
independently in all vertices of $\T_2$. In this model, that we call uniform rotation model, the rotor walk may be recurrent or transient, see Theorem \ref{thm:iidrotationT2}. However,  in higher dimensions, $\T_{d}, d\geq 3$, the walk is always transient, see Theorem \ref{thm:iidrotationTd}.
We conjecture that this behavior holds also for another model,  the uniform shift model, see Conjectures \ref{conj:shiftbinarytree} and \ref{conj:shiftdarytree}.
In particular, our results generalize recurrence and transience criteria for rotor-router walks on regular trees \cite{AnHo:11}, periodic trees \cite{HuSa:12}, and Galton-Watson trees with bounded degree \cite{HuMuSa:15}.

\subsection{Formal definition.}
Let $(\T_d,o)$ be the $d$-ary rooted tree, $d\ge 1$. For a vertex $v\in\T_d$,
we denote by $\pere{v}$ the parent of $v$ and by
$\fils{v}{1},\fils{v}{2},\ldots,\fils{v}{d}$ its children. We have in mind a planar embedding of the tree where the root is at the top, each generation is located below its ancestors,
and the children are ordered from left to right. We add a self-loop to the origin $o$ and define $o_{0}=o$.
A \emph{rotor sequence} is a function $\aseq:\N\to \{0,\ldots,d\}$.
A rotor sequence $\aseq$ is called \emph{non-degenerate} if for every $r\in \{0,\ldots,d\}$
there are infinitely many $x \in\N$ such that $\aseq(x)=r$. We denote by $\CA\subset\{0,\ldots,d\}^\N$ the set of all non-degenerate rotor sequences.

For $v\in \T_d$, we consider a rotor sequence $\aseq_v$ on $v$ by
identifying the set $\{0,\ldots,d\}$ with $N_v=\{\pere{v},\fils{v}{1},\fils{v}{2},\ldots,\fils{v}{d}\}$, 
the ordered (from left to right) set of $v$'s neighboring vertices.
We refer the reader to Figure \ref{fig:ex rr on T2} for an illustration of these definitions.

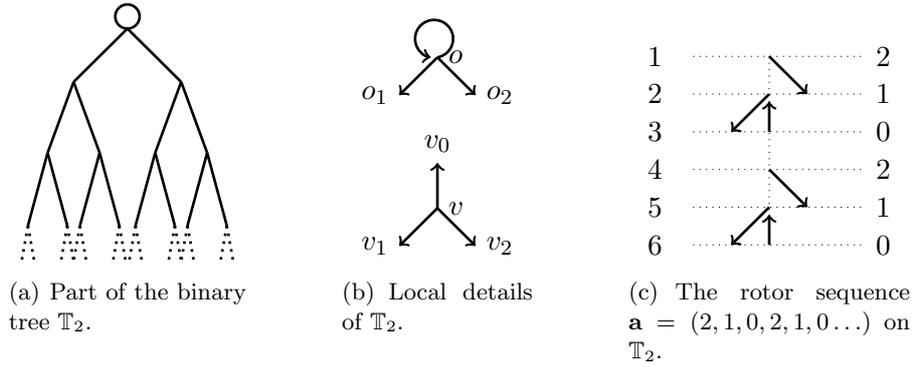
\begin{figure}
\centering
\subfigure[Part of the binary tree $\T_{2}$.]{
\begin{tikzpicture}[line width=1pt, scale=1]
\draw[-] (0,0) arc (-90:270:0.16);
\begin{scope}[shift={(0,0)},rotate=-135]
\draw[-] (0,0) -- (1,0) node (A){};
\end{scope}
\begin{scope}[shift={(0,0)},rotate=-45]
\draw[-] (0,0) -- (1,0) node (B){};
\end{scope}

\begin{scope}[shift={(A)},rotate=-110]
\draw[-] (0,0) -- (1,0) node (AA){};
\end{scope}

\begin{scope}[shift={(A)},rotate=-70]
\draw[-] (0,0) -- (1,0) node (AB){};
\end{scope}

\begin{scope}[shift={(B)},rotate=-110]
\draw[-] (0,0) -- (1,0) node (BA){};
\end{scope}

\begin{scope}[shift={(B)},rotate=-70]
\draw[-] (0,0) -- (1,0) node (BB){};
\end{scope}

\begin{scope}[shift={(AA)},rotate=-105]
\draw[-] (0,0) -- (1,0) node (AAA){};
\end{scope}

\begin{scope}[shift={(AA)},rotate=-75]
\draw[-] (0,0) -- (1,0) node (AAB){};
\end{scope}

\begin{scope}[shift={(AB)},rotate=-105]
\draw[-] (0,0) -- (1,0) node (ABA){};
\end{scope}

\begin{scope}[shift={(AB)},rotate=-75]
\draw[-] (0,0) -- (1,0) node (ABB){};
\end{scope}

\begin{scope}[shift={(BA)},rotate=-105]
\draw[-] (0,0) -- (1,0) node (BAA){};
\end{scope}

\begin{scope}[shift={(BA)},rotate=-75]
\draw[-] (0,0) -- (1,0) node (BAB){};
\end{scope}

\begin{scope}[shift={(BB)},rotate=-105]
\draw[-] (0,0) -- (1,0) node (BBA){};
\end{scope}

\begin{scope}[shift={(BB)},rotate=-75]
\draw[-] (0,0) -- (1,0) node (BBB){};
\end{scope}

\begin{scope}[shift={(AAA)},rotate=-100]
\draw[-, dotted] (0,0) -- (0.5,0) {};
\end{scope}

\begin{scope}[shift={(AAA)},rotate=-80]
\draw[-, dotted] (0,0) -- (0.5,0) {};
\end{scope}

\begin{scope}[shift={(AAB)},rotate=-100]
\draw[-, dotted] (0,0) -- (0.5,0) {};
\end{scope}

\begin{scope}[shift={(AAB)},rotate=-80]
\draw[-, dotted] (0,0) -- (0.5,0) {};
\end{scope}

\begin{scope}[shift={(ABA)},rotate=-100]
\draw[-, dotted] (0,0) -- (0.5,0) {};
\end{scope}

\begin{scope}[shift={(ABA)},rotate=-80]
\draw[-, dotted] (0,0) -- (0.5,0) {};
\end{scope}

\begin{scope}[shift={(ABB)},rotate=-100]
\draw[-, dotted] (0,0) -- (0.5,0) {};
\end{scope}

\begin{scope}[shift={(ABB)},rotate=-80]
\draw[-, dotted] (0,0) -- (0.5,0) {};
\end{scope}

\begin{scope}[shift={(BAA)},rotate=-100]
\draw[-, dotted] (0,0) -- (0.5,0) {};
\end{scope}

\begin{scope}[shift={(BAA)},rotate=-80]
\draw[-, dotted] (0,0) -- (0.5,0) {};
\end{scope}

\begin{scope}[shift={(BAB)},rotate=-100]
\draw[-, dotted] (0,0) -- (0.5,0) {};
\end{scope}

\begin{scope}[shift={(BAB)},rotate=-80]
\draw[-, dotted] (0,0) -- (0.5,0) {};
\end{scope}

\begin{scope}[shift={(BBA)},rotate=-100]
\draw[-, dotted] (0,0) -- (0.5,0) {};
\end{scope}

\begin{scope}[shift={(BBA)},rotate=-80]
\draw[-, dotted] (0,0) -- (0.5,0) {};
\end{scope}

\begin{scope}[shift={(BBB)},rotate=-100]
\draw[-, dotted] (0,0) -- (0.5,0) {};
\end{scope}

\begin{scope}[shift={(BBB)},rotate=-80]
\draw[-, dotted] (0,0) -- (0.5,0) {};
\end{scope}
\end{tikzpicture}
}
\hspace{1cm}
\subfigure[Local details of $\T_{2}$.]{
\begin{tikzpicture}[scale=0.5]
\begin{scope}[line width=1pt, yshift=2cm]
\draw[->] (-8,-2,0) -- (-7,-3,0);
\draw[->] (-8,-2,0) -- (-9,-3,0);
\draw[->] (-8,-2,0) arc (-80:260:0.5);
\node[below, right] at (-7,-3,0) {$\fils{o}{2}$};
\node[below, left] at (-9,-3,0) {$ \fils{o}{1}$};
\node[above, right] at (-8,-2,0) {$o$};
\end{scope}
\begin{scope}[line width=1pt, yshift=-2.cm]
\draw[->] (-8,-2,0) -- (-7,-3,0);
\draw[->] (-8,-2,0) -- (-9,-3,0);
\draw[->] (-8,-2,0) -- (-8,-0.8,0);
\node[below, right] at (-7,-3,0) {$\fils{v}{2}$};
\node[above] at(-8,-0.8,0) {$ \pere{v}$};
\node[below, left] at (-9,-3,0) {$ \fils{v}{1}$};
\node[above, right] at (-8,-2,0) {$v$};
\end{scope}
\end{tikzpicture}
}
\hspace{1cm}
\subfigure[The rotor sequence $\aseq=(2,1,0,2,1,0\ldots)$ on  $\T_{2}$.]{
\label{fig:1a}
\begin{tikzpicture}[scale=0.5]
\begin{scope}[line width=1pt]
\draw[thin, dotted] (0,0,0) -- (0,-5,0);
\foreach \y in {1,...,6} {\draw[thin, dotted] (-2,-\y+1,0) -- (2.5,-\y+1,0) ; \node at (-3,-\y+1,0) {$\y$};	}
\node at (3,1,0) {};
\node at (3,0,0) {$2$};
\node at (3,-1,0) {$1$};
\node at (3,-2,0) {$0$};
\node at (3,-3,0) {$2$};
\node at (3,-4,0) {$1$};
\node at (3,-5,0) {$0$};	

\node at (0,1,0) {};	
\draw[->] (0,0,0) -- (1,-1,0);
\draw[->] (0,-1,0) -- (-1,-2,0);
\draw[->] (0,-2,0) -- (0,-1.2,0);
\draw[->] (0,-3,0) -- (1,-4,0);
\draw[->] (0,-4,0) -- (-1,-5,0);
\draw[->] (0,-5,0) -- (0,-4.2,0);	
\end{scope}	
\end{tikzpicture}
}
\caption{The notations for the binary tree $\T_{2}$.}
\label{fig:ex rr on T2}
\end{figure}
To each vertex of $\T_{d}$ we assign a non-degenerate rotor sequence to get a rotor configuration 
$\{\aseq_v\}_{v\in\T_d}\in\CA^{\T_d}$. The corresponding \emph{rotor walk} $X=(X_n)_{n\ge 0}$ on $\T_d$, 
with a local time $\Cl=(\Cl_n)_{n\ge0}$ is defined recursively as follows:
$$X_0:=o,\,\, \Cl_0\equiv 0, X_{n+1}=\fils{(X_n)}{i}, \text{ and } \Cl_{n+1}(v)=\Cl_n(v) +  \delta_{X_n}(v),$$ where $i=\aseq_{X_n}(1+\Cl_n(X_n))$.
In words, the rotor walk starts in $o$ at time $n=0$. At time $n=1$ it follows the direction of the first rotor in the sequence $\aseq_{o}$, i.e.\ it moves to $o_{\aseq_{o}(1)}$.
Inductively, assume that the rotor walk is in $v$ at time $n$, then at time $n+1$ the walk
moves in the direction of the first unused rotor in $v$ and moves to $\fils{v}{i}$, where $i=\aseq_{v}(1+\Cl_n(v))$.

We define a \emph{finite excursion} of $X$ to be a finite sequence $X_m,X_{m+1},\ldots,X_{m+n}$ starting at the root 
$X_m=o$ and ending with a self loop $X_{m+n-1}=X_{m+n}=o$ such that it does not contain any self loop before that time, 
that is $X_j\ne o$ if $X_{j-1}=o$ for all $m+1\le j \le m+n-1$.
An \emph{infinite excursion} of $X$ is an infinite sequence $X_m,X_{m+1},\ldots$ starting at the root $X_m=o$ so that  
$X_j\ne o$ whenever $X_{j-1}=o$ for all $j\ge m+1$. 

Whenever the rotor configuration is in $\CA^{\T_d}$, the walk $X$ is recurrent if and only if every vertex is 
visited infinitely many times. Indeed, if the walk visits the origin infinitely many times, 
by non-degeneracy all of its neighbors are visited infinitely often; 
hence, by induction this is the case for all vertices of the graph. 
We shall assume throughout the paper that the rotor sequences are
always non-degenerate, even whenever it is not mentioned explicitly.

\begin{definition}[Finitely supported distribution model]\label{def:general}
Let $p$ be a probability measure on $\CA$ with finite support $S$. Enumerating the elements in $S$ from $1$ to $|S|$ we rewrite $p$ as $p=(p_{1},p_{2},\ldots, p_{|S|})$; the corresponding
rotor configurations are written as $\aseq^{(1)},\ldots, \aseq^{(|S|)}\in\CA^{\T_d}$. 
We call the rotor distribution a \emph{finitely supported distribution} if we sample a rotor sequence in every vertex independently and identically distributed according to $p$.
\end{definition}

The following two models are interesting special cases of the finitely supported distribution model.

\begin{definition}[Uniform rotation model]\label{def:i.i.d.uniform_rotation}
Fix a rotor sequence $\aseq\in \CA$.
Let $\pi\in S_{d+1}$ be the permutation on the $d+1$ symbols $\{0,\ldots,d\}$ defined by the rotation $\pi:n\mapsto n+1 \mod (d+1)$. We denote by $\pi\aseq$ the rotor sequence given by a point-wise use of $\pi$ on $\aseq$:
$\pi \aseq(x):=\pi (\aseq(x))$, $x\in\N$.
We call the rotor distribution a \emph{uniform rotation} if we sample a rotor sequence in every vertex independently and identically distributed according to the uniform distribution on the set  of all $d+1$ rotations of $\aseq$: $\{ \aseq, \pi\aseq, \pi^2\aseq, \ldots, \pi^{d}\aseq\}$.
\end{definition}

A rotor sequence on the $d$-ary tree $\aseq$ is said to be \emph{L-periodic} if $\aseq(x+L)=\aseq(x)$ for all $x\in \N$ for some $L$.
We say that an $L$-periodic rotor sequence $\aseq$ is balanced if each of the values $\{0,1,2,\ldots, d\}$ appears $N=L/(d+1)$ times in the first $L$ rotors.

\begin{definition}[Uniform shift model]\label{def:i.i.d.uniform_shift}
Let $\aseq^{(1)}=(a_{n})_{n\in \N}\in \CA$ be an $L$-periodic rotor sequence. Let $S^{(i)}: (a_{n})_{n\in \N} \mapsto (a_{n+i})_{n\in \N}$ be the shift operator and define $\aseq^{(i)}=S^{(i-1)}\aseq^{(1)}$ for $i\in \{2,\ldots, L\}$.
We call the rotor distribution a \emph{uniform shift} if we sample a rotor sequence in every vertex independently and identically distributed according to the uniform distribution on the set  all possible shifts of
$\aseq$: $\{ \aseq^{(1)}, \ldots, \aseq^{(L)}\}$.
\end{definition}


\section{The unary tree}
In this section we analyze rotor walks on the unary tree $\N$. We begin in Section \ref{subsec:Random rotation of a rotor sequence: recurrence} 
with the uniform rotation model and show recurrence. We then turn to periodic sequences and bound the number of infinite excursions by half the period in Section
\ref{subsubsec:Deterministic properties: the number of infinite excursions}. Finally, in Section
\ref{subsubsec:L-periodic sequences: a criterion for recurrence} 
we derive a simple criterion for recurrence of the walk. 
For convenience of the reader and in order to emphasize the structure of $\N=\T_{1}$ we will in the rest of this 
section use  $-1$ and $+1$ for $0$ and $1$, respectively, in the rotor sequences.
\subsection{Uniform rotations: recurrence}\label{subsec:Random rotation of a rotor sequence: recurrence}

We shall show that the walk is recurrent on the unary tree $\N$, whenever $\aseq$ is a non-degenerate rotor sequence, and the (two) permutations on $\aseq$ are chosen uniformly and independently in all vertices.


\begin{theorem}\label{thm:Recurrence in dimension 1}
Let $\aseq$ be a rotor sequence. Then, the rotor walk on $\N$ in the uniform rotation model corresponding to $\aseq$ (see Definition \ref{def:i.i.d.uniform_rotation})  is recurrent a.s.
\end{theorem}
Before we prove this theorem we introduce some more notations and basic facts.
For a rotor sequence $\aseq$ and $x\in\N$ we define 
\begin{equation}\label{def:T_functions}
T_\aseq(x)=\inf\left\{t\ge 0: \sum_{i=1}^t \mathbf{1} _{\aseq(i)=0}=x\right\},
\end{equation}
the number of rotors in $\aseq$ prior to $x$ rotors pointing to $0$. Define $U_\aseq(0):=0$ and
\begin{equation}\label{def:U_functions}
U_\aseq(x)=T_\aseq(x)-x
\end{equation}
be the number of $+1$-rotors in $\aseq$ prior to $x$ $-1$-rotors.
Call $U_\aseq$ \emph{the $U$-function} of $\aseq$.
\begin{lemma}\label{lem:subduality} \cite[Observation 2.8]{ABO} Let $\aseq$ be a rotor sequence. The following hold for the transposition $\tau$ of $-1$ and $+1$.
\begin{enumerate}
\item $U_\aseq(x)$ and $U_{\tau \aseq}(x)$ are monotonically non-decreasing in $x$.
\item $U_{\tau \aseq}(U_{\aseq} ( x ) ), U_\aseq(U_{\tau \aseq} ( x ) ) < x$ for every $x\in \N$, and equality holds if $x=0$.
\end{enumerate}
\end{lemma}

%
Fix $\{\aseq_v\}_{v\in\N}\in\CA^{\N}$ and $k\in\N$. We are now interested to know whether $X$ has $k$ finite excursions. 
For that we define a process $Z^k$ on the non-negative integers inductively by the following:
\begin{equation}\label{eq:def_of_Z_on_Unary}
Z^k_0:=k \text{ and } Z^k_{n}=U_{\aseq_{n-1}}(Z^k_{n-1}),\quad n\ge 1.
\end{equation}
Note that $Z^k_n$ is non-decreasing in $k$ by Lemma \ref{lem:subduality} (1), but not necessarily in $n$.
We also write $Z_{n}$ for $Z_{n}^{1}$. Here is an informal description of $Z^k$. The variable $Z_{1}^{k}$ is the number 
of rotors in $\aseq_{0}$ pointing away from $0$ (i.e., to the right) prior to $k$ rotors pointing back to $0$ 
(i.e., to the left). In other words, the $k$th excursion is finite if and only if the walk  jumps from $1$ to 
$0$ exactly $Z_{1}^{k}$ times. Assuming the $(k-1)$st excursion is finite, by induction on $n$, the $k$th excursion is finite 
if and only if for every vertex $n$ the walk jumps exactly a total number $Z_{n}^{k}$ times from $n$ to $n-1$ 
in the first $k$ excursions (which can be infinite time a prior).   Eventually, $Z_{n}^{k}$ is the number of times the vertex $n-1$ is visited from $n$ before the 
$k$th traversal of the loop in the root. 

The following Lemma is a now evident. It can also be considered as a consequence of Proposition 3.4 in \cite{BS2009recurrence} together with the deterministic point of view of Chapters 2.2 and 2.3 of \cite{ABO}.
\begin{lemma}\label{lem:rec_rw}  The first $k$ excursions of $X$ are (well defined and) finite if and only if $Z^k_n=0$ for some $n$.
\end{lemma}

We now can prove Theorem \ref{thm:Recurrence in dimension 1}.

\begin{proof}[Proof of Theorem \ref{thm:Recurrence in dimension 1}]
Recurrence follows once we show that all excursions of $X$ are finite a.s.\
For notation, write $S_2=\{id,\pi\}$. For each rotor configuration $\{\aseq_v\}_{v\in\T_d}\in\CA^{\T_d}$ we let 
$Y_n:=\sign (\sigma_{n-1})$, where $\sigma_n\in S_2$ is such that
$\aseq_{n}=\sigma_{n} \aseq$, $n\ge 0$. Define $Q_0=0$ and $Q_n=\sum_{1=1}^n Y_i$.
Then the measure on $Q$ is distributed as a simple symmetric random walk on $\Z$, started at the origin.
For every $k\in \N$ we shall show that the $k$th excursion is well defined and finite a.s.
By Lemma \ref{lem:rec_rw}, it is enough to show that the process $Z^k_n=0$ for some finite time $n>0$ a.s. 
Since $Z_{n}^{k}$ is a Markov chain it follows by repeated use of the strong Markov property that it is enough 
to show that $Z^k_n$ gets below $k-1$ in finite time a.s.
Let $N$ be the first random positive integer $n\ge 1$ such that $Q_n=0$. We claim that $Z^k_{N}\le k-1$.
In other words, our goal is to show that the following decomposition of $N$ functions
\begin{equation}\label{eq:Znk}
U_{\sigma_{N-1} \aseq} \circ \cdots \circ  U_{\sigma_{2} \aseq} \circ U_{\sigma_{0} \aseq}(k)
\end{equation}
is bounded above by $k-1$. We first sketch the proof. By the definition of $N$, there are exactly $N/2$ appearances of $U_{\pi \aseq}$ and $N/2$ appearances of $U_{ \aseq}$ in the last decomposition. 
Therefore, one can regard the ordered sequence $(Y_N,\ldots,Y_1)$ as balanced parentheses where the rightmost sign corresponds to 
`$)$'  and the opposite sign corresponds to  `$($'. 
Then, using Lemma \ref{lem:subduality} we iteratively bound our expression by removing  $U_{\sigma_{j+1}\aseq}\circ U_{\sigma_{j}\aseq}$ that corresponds to the right-most couple `$()$'.  In the last iteration, we use Lemma \ref{lem:subduality} to conclude.    

More formally, let $J$ be the first $j$ so that $Y_j\ne Y_{j+1}$.
Due to Lemma \ref{lem:subduality} every factor of the form $U_{\pi \aseq} \circ U_{ \aseq} (y) $ or 
$U_{ \aseq} \circ U_{\pi \aseq}(y)$ is at most $y-1$.
In particular,  
\begin{equation*}
U_{\sigma_{J+1} \aseq} \circ U_{\sigma_{J} \aseq} \circ U_{\sigma_{J-1} \aseq} \circ \cdots \circ U_{\sigma_{1} \aseq}(k)\le
U_{\sigma_{J-1} \aseq} \circ \cdots \circ U_{\sigma_{1} \aseq}(k).
\end{equation*}
By Lemma \ref{lem:subduality}(1) every decomposition of some of the
functions $U_{\sigma_{j} \aseq}$ is also monotone. 
Therefore, the whole expression in \eqref{eq:Znk} is bounded by  
\begin{equation*}
U_{\sigma_{N} \aseq} \circ \cdots \circ  U_{\sigma_{J+2} \aseq} \circ U_{\sigma_{J-1} \aseq} \circ \cdots \circ U_{\sigma_{1} \aseq}(k)\le
U_{\sigma_{N} \aseq} \circ \cdots \circ  U_{\sigma_{J+2} \aseq} \circ U_{\sigma_{J-1} \aseq} \circ \cdots \circ U_{\sigma_{1} \aseq}(k).
\end{equation*}
By considering the right end side of the last inequality instead of \eqref{eq:Znk}, after a suitable update of the indices, we can iterate 
the argument above $N/2-1$ times to get 
\begin{equation*}
Z^k_N=U_{\sigma_{N-1} \aseq} \circ \cdots \circ  U_{\sigma_{2} \aseq} \circ U_{\sigma_{0} \aseq}(k) \le 
U_{\sigma_{j+m} \aseq}  \circ U_{\sigma_{j} \aseq}(k).
\end{equation*}
for some $j,m\ge 1$ with different signs $Y_j\ne Y_m$. Using Lemma \ref{lem:subduality} again, the last expression is bounded above by $k-1$.
\end{proof}

\subsection{Periodic balanced configurations: a criterion}\label{subsec:criterion for periodic balanced rotor configurations}

In this section we consider configurations on the unary tree where for each vertex there is an $L$-periodic and balanced rotor sequence.
\subsubsection{Deterministic properties.}\label{subsubsec:Deterministic properties: the number of infinite excursions}

The following properties of the $U$-functions (see the definition in \eqref{def:U_functions}) are immediate consequences of periodicity together with being balanced.
\begin{lemma}\label{lemma:prop_U}
Let $\aseq$ be an $L$-periodic balanced sequence, $U$ its $U$-function, and $N=L/2$. Then,
for $x=\alpha N + \beta$ with $\alpha,\beta\in\N$ and $1\leq \beta \leq  N$ we have that
\begin{equation*}
U(x)=\alpha N + U(\beta) \mbox{ and } U(\beta)\leq N.
\end{equation*}
\end{lemma}

For the proof of the next theorem we consider the notion of \emph{leftover environments}. We start with  a rotor configuration $\{\aseq_{n}\}_{n\in\N}$.
If the rotor walk is transient then  the total number of rotors $l(n)$ used in $n$ by the walk is finite. The rotor configuration after the first infinite excursion is well-defined and noted by $LO$. 
In other words, $LO(n,i):=\aseq_{n}(l(x)+i),\ n\in \N$.
Now another walker starts from the origin and moves according to the rotors in the leftover configuration. This procedure can be sequentially executed as long as the walkers are transient.

\begin{theorem}[Number of infinite excursions]\label{theorem:bounded infinite excursion N}
Let $L\in\N$. For any vertex $n\in \N$ choose (deterministically) an $L$-periodic and balanced rotor sequence $\aseq_{n}$. Then, the number of infinite excursions is bounded from above by $L/2$.
\end{theorem}
\begin{proof}
Let us first give an intuition for the proof. For each infinite excursion the number of right crossings of an edge is the number of left 
crossing of this edge plus $1$. Assume now that there have been $N=L/2$ infinite excursions and denote by $R$ 
the right-most position of all finite excursions prior to this time. For each vertex $n>R$ the number of $+1$'s consumed in 
$n$ is the number of $-1$ used in $n+1$ plus $N$. This means that the leftover environment has a ``local drift'' to the left for 
all $n>R$. This turns out to be enough to prevent the next excursion to be infinite. By induction, all subsequent excursions will be finite as well. 

To make this argument precise denote by $U_{n}^{(N)}$ the $U$-function in vertex $n$ corresponding to the leftover
configuration after the $N$th infinite excursion. In other words, if $l^N(n)$ is the total number of rotors that were used by the first $N$ walkers
(see the paragraph before  Theorem \ref{theorem:bounded infinite excursion N}) then $U$ corresponds to the rotor configuration $\aseq^N_{n}$ which is given by $\aseq_{n}^N(i):=\aseq_{n}(i+l^N(x))$.

The $U$-functions $U_{n}$ for $n> R$ are not changed by the
finite excursions. In order to control the changes made by the infinite excursions we denote by $\ell_{n}$ (resp.\ $r_{n}$) the total number of left (resp.\ right) rotors
used by the rotor walk in $n$ after the $N$th infinite excursion. For all $n>R$ and every $x\in\N$ we now have
\begin{equation}\label{eq:U-f-n_after_N_excursions}
U_{n}^{(N)}(x)=U_{n}(x+\ell_{n})-r_{n}=U_{n}(x+\ell_{n})-\ell_{n+1}-N,
\end{equation}
since $r_n=\ell_{n+1}+N$. 
After the $N$th infinite excursion we start another rotor walk from the origin and show that its  excursion is finite. Note that whenever a configuration is $L$-periodic (and balanced) then
so is its leftover configuration. Therefore, due to Lemma \ref{lemma:prop_U} we have that $Z_{R}\leq N$, where $Z$ is the process defined in \eqref{eq:def_of_Z_on_Unary}
corresponding to the $U_n^{(N)}$'s. There are two cases: either the
rotor walk never visit $R$, or it does.
In the first case the excursion is finite by non-degeneracy. Hence it remains to consider the case $Z_{R}>0$.
From \eqref{eq:U-f-n_after_N_excursions} we have in this case that
\[
Z_{R+1}=U_{R+1}^{N}(Z_{R})=U_{R+1}( Z_{R}+\ell_{R+1})-\ell_{R+2}-N.
\]
Let $\alpha\in\N$ and $1\leq \beta\leq N$ such that
$$Z_{R}+\ell_{R+1}=\alpha N +\beta.$$
By Lemma \ref{lemma:prop_U} and \eqref{eq:U-f-n_after_N_excursions},
\begin{equation}\label{eq:Z_R+1}
Z_{R+1}=U_{R}(\alpha N +\beta)-\ell_{R+2}-N =(\alpha-1) N + U_{R}(\beta)-\ell_{R+2}.
\end{equation}
Now,
\begin{eqnarray*}
Z_{R+2} &=& U^{(N)}_{R+2} (Z_{R+1})\\
&=& U_{R+2}(Z_{R+1}+\ell_{R+2}) - \ell_{R+3}-N\\
&=& U_{R+2}((\alpha-1) N + U_{R+1}(\beta)-\ell_{R+2} +\ell_{R+2}) - \ell_{R+3}-N\\
&=& (\alpha-2) N + U_{R+2}\circ U_{R+1}(\beta) - \ell_{R+3},
\end{eqnarray*}
where we assume that all the terms are non-negative (otherwise $Z_{R+2}$ equals zero by definition), the second equality uses \eqref{eq:U-f-n_after_N_excursions}, and the third and fourth equalities follow from Lemma \ref{lemma:prop_U}.
Inductively, as long as  $Z_{R+j}>0$ we have
 \begin{eqnarray*}
Z_{R+j+1} &=& U^{(N)}_{R+j+1} (Z_{R+j})\\
&=& (\alpha-j-1) N + U_{R+j+1}\circ \cdots \circ U_{R+2}\circ U_{R+1}(\beta) - \ell_{R+j+2}.
\end{eqnarray*}
At each step the quantity $Z_{R+j}$ is reduced by at least $N$, hence this procedure stops after a finite number of steps. In other words $Z_{R+j+1}=0$ for some $j\le \alpha$. 
It follows by induction that every other excursion has to be finite too. Indeed, modifying the definition of $R$ to include all previous finite excursions, the argument is the same.
\end{proof}

\begin{example}\label{ex:bd inf exc}
The bound obtained in Theorem \ref{theorem:bounded infinite excursion N} is sharp. Indeed, consider the case where all sequences equal the $2$-periodic sequence $(+1,-1,\ldots)$.
Then, the first excursion to the right is infinite, but all subsequent excursions are finite, see also Figure \ref{fig:ex:bd inf exc}.

\begin{figure}
\centering
\begin{tikzpicture}[scale=0.5]
\foreach \x in {0,2,...,10}
{ \draw[thin, dotted] (\x,0) -- (\x,-7);
\foreach \y in {0,-2,-4,-6}{\draw[black,line width=1pt,->] (\x,\y) -- (\x+1,\y);	}
\foreach \y in {-1,-3,-5,-7}{\draw[black,line width=1pt,->] (\x,\y) -- (\x-1,\y);	}
}
\begin{scope}[line width=1pt]
\draw[black, ->] (0,0.2) -- (11,0.2);
\draw[gray, ->] (0,-0.8) -- (0.5,-0.8) -- (0.5,-1.2) -- (0,-1.2);
\draw[gray, ->]  (0,-1.8)  -- (1.5,-0.8) -- (2.5, -0.8) -- (2.5,-1.2) --  (0, -3.2);
\draw[gray, ->] (0,-3.8)   -- (3.5,-0.8) -- (4.5,-0.8) -- (4.5,-1.2) -- (0,-5.2);
\draw[gray, ->] (0,-5.8)   -- (5.5,-0.8) -- (6.5,-0.8) -- (6.5,-1.2) -- (0,-7.2);
\end{scope}
\end{tikzpicture}
\caption{One infinite (black) excursion and the following four finite (gray) excursions in Example \ref{ex:bd inf exc}.}
\label{fig:ex:bd inf exc}
\end{figure}

\end{example}

\begin{remark} On $\Z$, in the case where the same rotor sequence is in all the vertices, it is possible to have infinite excursions to both directions, left and right.
Indeed, take $(+1,-1,-1,+1,\ldots)$, then $U(1)=U(2)=1$ so $Z^1$ survives. But also in the negative direction, we have $U(1)=0$ but
$U(2)=2$, and so $Z^2$ survives. Hence, here the first excursion is infinite (to the right), then we have a finite left excursion, and then an
infinite left excursion. 
\end{remark}

\subsubsection{\texorpdfstring{$L$}{L}-periodic sequences: a criterion for recurrence}\label{subsubsec:L-periodic sequences: a criterion for recurrence}

Let $\aseq^{(1)}, \aseq^{(2)},\ldots, \aseq_{\ell}$ be $L$-periodic balanced sequences. Denote by $U_{i}$ the $U$-function corresponding to the rotor configuration $\aseq_{i}$,
$i=1,\ldots,\ell$.
Let $p=(p_{1},p_{2},\ldots,p_{\ell})$ be a strictly positive probability vector: $p_i>0$ for all $i$ and $\sum_{i=1}^\ell p_i =1$. In the following we consider the rotor configuration with distribution $p$, see Definition  \ref{def:general}.
Define
\begin{equation}
k^{*}=\inf\{k\geq 1: \forall i:U_{i}(k)\geq k\}
\end{equation}
with the convention that $\inf \emptyset=\infty$.

\begin{theorem}\label{theorem:rec trans Z}
If $k^{*}<\infty$ then the walk is transient a.s.\ and otherwise it is recurrent a.s.
\end{theorem}

\begin{proof}
If $k^{*}<\infty$ then $Z_{n}^{k^{*}}\geq k^{*}$ for all $n\geq 1$ and by Lemma \ref{lem:rec_rw} the $k$th excursion is infinite for some $k\le k^*$, and the walk is transient.
For the other direction, assume that $k^{*}=\infty$. Let $x=\alpha N+\beta$ with $\alpha,\beta\in \N$ such that $1\leq \beta\leq N$. The theorem follows once we show that the $x$th excursion is finite for any choice of $x$.
By Lemma \ref{lem:rec_rw}, it is enough to show that a.s.\ $Z_{n}^x=0$ for some $n$. By Lemma \ref{lemma:prop_U}, $Z_{n}^{x}\leq (\alpha+1) N$ for all $n\geq 1$, and moreover, since $k^*=\infty$
then for every $k\in\{1,\ldots,N\}$ there is some $i_k\in\{1,\ldots,\ell \}$ such that $U_{i_k}(k)\le k-1$.
By monotonicity of the $U$-functions we have $$U_{i_1} \circ U_{i_2} \circ\ldots \circ U_{i_N}(N)\le U_{i_1} \circ U_{i_2} \circ\cdots \circ U_{i_{N-1}}(N-1)\le\ldots\le U_{i_1}(1)=0.$$
Using this together with Lemma \ref{lemma:prop_U} we have that $Z^x_{m+N}\le \alpha N$ if $(\aseq_{m},\ldots,\aseq_{m+N}) = (\aseq_{i_N},\ldots,\aseq_{i_1} )$.
Since the last event has a positive probability, by the i.i.d.\ assumption and the Borel-Cantelli Lemma we have that there are infinitely many such $m$'s a.s., denote them by $m_1,m_2,\ldots$.  Eventually, we have $Z^x_{m_\alpha+N}=0$.
\end{proof}


\begin{example}[Rotor-router walk]
Consider the $2$-periodic sequences $\aseq^{(1)}=(-1,+1,\ldots)$ and $\aseq^{(2)}=(+1,-1,\ldots)$, and let $p=(p_{1},1-p_{1})$ with $p_{1}\in(0,1)$. Then, $U_{1}(x)=x-1$ and $U_{2}(x)=x$ for all $x\in \N$ and hence $k^{*}=\infty$ which implies that the rotor walk is recurrent a.s.
\end{example}

\begin{corollary}[Shifts of a balanced sequence]\label{cor:random_starting_point_is_recurrent}
Let $\aseq$ be an $L$-periodic balanced sequence. Choose an i.i.d.\ configuration such that each of its shifts has a strictly positive probability. Then, the rotor walk is recurrent a.s.
\end{corollary}

\begin{lemma}[]\label{lem:Cycle}
Assume that $a_i\in\{-1,1\}, i=1,2,\ldots,2n$ are such that
\begin{equation}\label{eq:nonPositiveSequence}
\sum_{i=1}^{2n} a_i=0.\end{equation}
For $k>2n$ define $a_k=a_i$ whenever $i\equiv k\mod 2n$. Then there is a starting point $j\in\{1,2,\ldots,2n\}$ so that $\sum_{i=j}^{j-1+k}a_i \le 0$ for all $k\ge 1$.
\end{lemma}
One proof of the Lemma is based on a simple variation of the well known Cycle Lemma \cite{dvoretzky1947problem}. We shall supply an even shorter proof due to one of the referees.

\begin{proof}
Consider the $a_{i}$ as the increments of a nearest neighbor walk path. The index at which the walk reaches a maximum gives the desired starting point.
\end{proof}
The short proof of the next Lemma is due to the same referee as above.
\begin{lemma}\label{lem:strict_ineq_of_U}
Assume that $\aseq=(a_{1}, a_{2}, \ldots)$ with $a_i\in\{-1,1\}$  such that $\sum_{i=1}^{k}a_i \le 0$ for all $k\ge 0$. Let $U$ be the corresponding $U$-function. Then $U(x)<x$ for all $x>0$.
\end{lemma}
\begin{proof}
Remember that $T_\aseq(x)$ is the location of the $x$th -1 (see \eqref{def:T_functions}). Using the assumption for $k=T_\aseq(x)-1$ we have 
$0\geq \sum_{i=1}^{T_\aseq(x)-1}a_{i}=U(x)-(x-1)$.
\end{proof}

\begin{proof}[Proof of Corollary \ref{cor:random_starting_point_is_recurrent}]
Lemma \ref{lem:Cycle} gives us a starting point $j$ such that the equality in the lemma holds. Then by Lemma \ref{lem:strict_ineq_of_U} for $\aseq^{(j-1)}= S^{j-1}\aseq$, it holds that
$U_{\aseq^{j-1}}(x)<x$ for all $x\ge 1$. Therefore $k^*=\infty$ and by Theorem \ref{theorem:rec trans Z} the rotor walk is recurrent a.s.
\end{proof}

\begin{theorem}(Shift model: the sequence is balanced if and only if the walk is recurrent)\label{thm:random_starting_point_non-blanced_is_transient}
Let $\aseq$ be an $L$-periodic sequence. Choose an i.i.d.\ configuration such that each of the shifts has a strictly positive probability. If the number of $+1$'s per period is strictly larger than the numbers of $-1$'s,
then the rotor walk is transient a.s. Otherwise it is recurrent a.s.
\end{theorem}
\begin{proof}
If the number of $-1$'s per period is larger than the number of $+1$'s, we can compare the sequence with a periodic sequence by arbitrarily changing a few $-1$'s to $+1$'s in the period
so that the resulted sequence is balanced (if $L$ is odd, look at $2L$ instead). Note that the $U$-function corresponding to the original sequence is not more than the one corresponding to the new one.
Using the proof of Corollary \ref{cor:random_starting_point_is_recurrent} for the resulted sequence, we get that there is at least one $U$-function $U$ such that $U(x)<x$ for all $x>0$.
Therefore $k^*=\infty$ and by Theorem \ref{theorem:rec trans Z} the rotor walk is recurrent a.s.

For the other case, denote by $\nu$ the number of $+1$'s per period, and by $\zeta=L-\nu$ the number of $-1$'s per period. We assume that $\nu>\zeta$. Let $U$ be the $U$-function corresponding to an arbitrary (but fixed) shift.
Observe that by definition of $U$, $$U(\zeta + 1)\ge \# \text{$1$'s in the first period}=\nu\ge\zeta+1.$$ Therefore $k^*\le\zeta+1<\infty$ and by Theorem \ref{theorem:rec trans Z} the rotor walk is transient a.s.
\end{proof}

\begin{remark}
 By considering right excursions and left excursions separately, the theorems in this chapter gives corresponding results on rotor walks on $\Z$.
 In particular, for a periodic sequence, choosing each one of the shifts with positive probability, we get that the rotor walk on $\Z$ is recurrent a.s.\ if the sequence is balanced, transient to the right a.s.\ if there are strictly larger numbers of $+1$'s than $-1$'s
 per period, and transient to the left a.s.\ otherwise.
\end{remark}

\section{\texorpdfstring{The $d$-ary tree, $d\geq 2$}{The regular tree}}\label{sec:d-arry}
In this section we consider rotor walks on the homogeneous tree $\T_{d}, d\geq 2$. We give a criterion for recurrence and transience, see Theorem \ref{theorem:bal rc dary}. This criterion is based on
generalizations of the $U$-functions defined in the previous section and the fact that
transience of the rotor walk is equivalent to survival of a multi-type Galton-Watson process, as we shall see below.
We then show  that the uniform rotation configuration may be recurrent or transient on $\T_{2}$ and classify the recurrent rotor sequences, see Theorem \ref{thm:iidrotationT2}.
In higher dimensions, $\T_{d}, d\geq 3$, the uniform rotation configuration is always  transient, see Theorem \ref{thm:iidrotationTd}.
We conjecture that a similar behavior holds also for the shift configuration, see Conjectures \ref{conj:shiftbinarytree} and \ref{conj:shiftdarytree}.

\subsection{Notations}
For a rotor sequence $\aseq$ on the homogeneous  tree $\T_{d}, d\geq 2$, i.e.\ $\aseq$ is taking values in $\{0,1,2,\ldots, d\}^{\N}$, we define the $U$-functions as follows:
\[
U_{\aseq}(x)=
\left(
\begin{array}{c}
 U_{\aseq}^{(1)}(x)    \\
  U_{\aseq}^{(2)}(x) \\
  \vdots\\
  U_{\aseq}^{(d)}(x) \\
  \end{array}
\right),
\]
where $U_{\aseq}^{(i)}(x)$ is the number of $i$'s that appear prior to $x$ $0$'s in the sequence $\aseq$.
For a finite sequence of rotors $\{a_1,\ldots,a_\ell\}$ of length $\ell$ we write $\period{a_1,\ldots,a_\ell}$ for the periodic rotor sequence $(a_1,\ldots, a_\ell, a_1,\ldots, a_\ell,\ldots)$.
\begin{example}[rotor-router walk]\label{ex:standrad rr binary tree1}
On $\T_{2}$ let $\aseq^{(1)}=\period{0,1,2}$, $\aseq^{(2)}=\period{1,2,0}$ and $\aseq^{(3)}=\period{2,0,1}$. The corresponding $U$-functions are given by
\[
U_{1}(x)=\left(
\begin{array}{c}
  x-1     \\
  x-1
\end{array}
\right),
U_{2}(x)=\left(
\begin{array}{c}
  x     \\
  x
\end{array}
\right)\mbox{ and }
U_{3}(x)=\left(
\begin{array}{c}
  x-1     \\
  x
\end{array}
\right).
\]
\end{example}

\begin{example}\label{ex:second on binary1}
On $\T_{2}$ we consider $\aseq^{(1)}=\period{0,1,0,1,2,2}, ~\aseq^{(2)}=\period{1,2,1,2,0,0}$ and $\aseq^{(3)}=\period{2,0,2,0,1,1}$. The corresponding $U$-functions are given by
\[
U_{1}(x)=\left(
\begin{array}{c}
  x-1     \\
  2 \lfloor \frac{x-1}{2} \rfloor
\end{array}
\right),
U_{2}(x)=\left(
\begin{array}{c}
   2 \lceil \frac{x-1}{2} \rceil    \\
   2 \lceil \frac{x-1}{2} \rceil
\end{array}
\right)\mbox{ and }
U_{3}(x)=\left(
\begin{array}{c}
   2 \lfloor \frac{x-1}{2} \rfloor      \\
   2 \lceil \frac{x-1}{2} \rceil
\end{array}
\right).
\]
\end{example}


Given $(\aseq_{v})_{v\in \T_{2}}$ we define a process $Z^k$ on the tree $\T_{2}$  by the following: $Z^k_o:=k$, and inductively for  $v=\fils{\pere{v}}{i}, i\in\{1,\ldots,d\}$,
\begin{equation}\label{eq:def Zkv}
Z^k_{v}= U_{\aseq_{v_{0}}}^{(i)}(Z_{\pere{v}}^{k}).
\end{equation}
Recall that $v=\fils{\pere{v}}{i}$ means that $v$ is the $i$th child of its parent $v_{0}$. 
This is a generalization of \eqref{eq:def_of_Z_on_Unary}. 
In words, $Z_{o_{i}}^{k}$ is the number jumps from $o$ to $o_{i}$ before the $k$th traversal of the loop in $o$, 
that is before the end of the $k$th excursion.
Inductively, $Z_{v}^{k}$ is the number of times the vertex $v_{0}$ is visited coming from its children $v$ before the 
$k$th traversal of the loop in the root. 

Similarly to Lemma \ref{lem:rec_rw} in the case of the unary tree, the following Lemma is a consequence of Proposition 3.4 in \cite{BS2009recurrence}, together with the deterministic point of view of Chapters 2.2 and 2.3 of \cite{ABO}. 

\begin{lemma}\label{lemma:rec_brw}  The first $k$th excursions of $X$ are finite if and only if $Z^k_v>0$ for only a finite number of vertices $v$.
\end{lemma}

Let us expand on the relation of $Z^{k}_{v}$ to a multi-type Galton-Watson process.  
The \emph{type} of a vertex $v$ is defined as $Z^{k}_{v}$. Given the rotor sequence $\aseq_{v}(\cdot)$ the types of $\fils{v}{i}, i\in\{1,\ldots,d\},$ are defined deterministically by Equation \eqref{eq:def Zkv}.
In the finitely supported distribution model, see Definition \ref{def:general}, we can give a more probabilistic description. The starting point of the construction is the $d$-ary tree $\T_{d}$, which itself can be seen as the genealogical tree of the branching process where each particle (vertex) has a.s.\ $d$ offspring particles (vertices).   To each vertex $v$ in $\T_{d}$ we assign inductively a type. To start let $k$ be the type of the root $o$. In other words,
we start the branching process with one particle of type $k$ at time $0$. Now, we choose a rotor configuration for the root at random (according to $p$) and the types of the children
$\fils{o}{1},\ldots \fils{o}{d}$ of $o$ are given by the values $Z_{\fils{o}{1}}^{k},\ldots, Z_{\fils{o}{d}}^{k}$ following Equation \eqref{eq:def Zkv}.
If the type of a particle is $0$ we declare the particle and all its descendants as dead. By induction this procedure  either continues until all particles are dead, i.e.\ the process dies out, or continues indefinitely,
i.e.\ the process survives. We denote by $\xi_{n}^{k}(i)$ the number of particles of type $i$ at time $n$ and write $\xi^{k}$ for the whole branching process. Due to the definition of
$\xi^{k}$ we have that $\xi^{k}$ dies out if and only if $Z_{v}^{k}>0$ for only a finite number of $v$.

For $v\in\T_d$ we denote by $|v|$ the level of $v$, i.e.\ its graph distance from the root $o$. The next lemma guarantees that in order to prove transience of the rotor walk for i.i.d.\ configurations it is sufficient that the process  $Z^{k}_{v}$ survives with positive probability for some $k$.
\begin{lemma}[]\label{lem:Positive probability for escaping suffices for transience}
Assume that the rotor configuration is i.i.d. Then, $$\pr_o[X_n\ne o ~\forall n>0]>0 \quad \Longrightarrow \quad \pr_o[X_n\ne o \text{ for all $n$ large enough}]=1.$$
\end{lemma}
\begin{proof} The proof is an adaptation of the proof of  \cite[Lemma 8]{kosygina2008positively}.  Let $K_i\in\N\cup\{\infty\}$ be the largest level passed in the $i$th excursion, defined to be $\infty$ if either the excursion is infinite,
or if $K_{i-1}=\infty$. Let $J\subset\N$ be
the set of all `tanned' indices, i.e.\ indices $j$ such that $K_j<\infty$ and $K_j>K_i$ for all $i<j$. Denote $p=\pr_o[X_n\ne o ~\forall n>0]>0$.
 By the i.i.d.\ assumption, for every $j\in J$ we have $\pr_{|X_0|=j}[|X_n|> j ~\forall n>0]=p>0$ independently of the past of the walk.
Therefore, $J$ is stochastically dominated by a geometric random variable (with parameter $p$) and therefore a.s.\ finite. This implies that on the event $\{X_n=o \text{ i.o.}\}$
the range $\{X_{n}, n\in \N\}$ is finite. By non-degeneracy, the latter happens with probability zero. Therefore $\pr_o[X_n=o \text{ i.o.}]=0$.
\end{proof}

\subsection{Periodic balanced configurations}\label{subsec:Periodic balanced rotor configurations}
Let us begin with an immediate but important property of periodic balanced rotor sequences, generalizing Lemma \ref{lemma:prop_U} for any $d$-ary tree.

\begin{lemma}\label{lem:per bal bound U on tree}
Let $\aseq$ be a periodic and balanced rotor configuration $\aseq$ with period $L=(d+1)N$ and let $U_{\aseq}=(U_{\aseq}^{(1)},\ldots, U_{\aseq}^{(d)}) $ be its $U$-function. Then,
for $x=\alpha N + \beta$ where $\alpha\in\N$, $1\leq \beta\leq  N$ and for $i\in\{1,\ldots, d\}$ we have that
\begin{equation*}
U_{\aseq}^{(i)}(x)=\alpha N + U_{\aseq}^{(i)}(\beta) \mbox{ and }  U_{\aseq}^{(i)}(\beta)\leq N.
\end{equation*}
\end{lemma}

A rotor configuration $\{\aseq_v\}_{v\in\T_d}\in\CA^{\T_d}$ is called $L$-periodic (balanced) if all the rotor sequences $\aseq_v,\,v\in\T_d$ are $L$-periodic (balanced).

We consider the finitely supported distribution model, see Definition \ref{def:general}. A consequence of Lemma \ref{lem:per bal bound U on tree} is that, in the periodic and balanced setting, $Z^{k}_{v}$ can be seen as a finite-type Galton-Watson process.
In fact, let  $k=\alpha N + \beta$ where $\alpha\in\N$, $1\leq \beta\leq  N$ and for $i\in\{1,\ldots, d\}$, then $\xi_{n}^{k}(y)=0$ (a.s.) for all $n\in \N$ and all $y>(\alpha+1)N$.

\begin{theorem}[Periodic balanced rotor configuration on  $\T_{d}, {d\geq 3}$]\label{theorem:bal rc dary}
In the above setting the rotor walk is transient a.s.\ if and only if the process $\xi^{N}$ survives with positive probability.
\end{theorem}
\begin{proof}
Due to Lemma \ref{lemma:rec_brw} and the discussion above it remains  to prove that the process $\xi^{k}$ dies out a.s.\  for all $k$ if and only if it dies out a.s.\ for  $k=N$.
Let $k=\alpha N +\beta$ with $1\leq \beta \leq N$ and assume that the process  $\xi^{N}$  dies out a.s. The latter together with Lemma \ref{lem:per bal bound U on tree} implies
that a.s.\ there exists a (random) level $n_{1}$ such that all particles at generation $n_{1}$ have a type of at most $\alpha_{1} N +\beta_{1}$ with $\alpha_{1}=\alpha-1$
and $1\leq \beta_{1} \leq N$. By induction, a.s.\ there exists a (random) level $n_{\alpha}$ such that all vertices will have type of at most $N$. By the assumption that
$\xi^{N}$  dies out a.s.\ all of these particles will a.s.\ have only a finite number of descendants.
\end{proof}

It is well known, see e.g.\ \cite[Chapter V]{AN:72} that survival of a multi-type Galton-Watson process only depends on the first moment matrix $M$ of
the process. Usually one assumes the  multi-type Galton-Watson process to be non-singular and positive regular. A branching process is called singular if every particle has exactly one offspring. In this paper all non-trivial examples 
give rise to non-singular processes, particularly in the finitely supported distribution model with $|S|>1$.  Positive regularity, i.e.\ strict 
positivity of the first moment matrix, is a condition that is not always satisfied in our cases. However, using standard theory of non-negative matrices, e.g.\ \cite[Chapter 1]{Sen:06}, it follows that the multi-type branching processes $\xi^{N}$ defined above will survive with positive probability if and only if the  largest eigenvalue of the first moment matrix  is strictly larger than $1$. 

In our case the first moment matrix is given as follows. Assume that $v$ has type $k$, then the mean number $m(k,\ell)$ of offspring of $v$ of type $\ell$ is given by
\begin{equation}\label{def:M}
m(k,\ell)=\sum_{i=1}^{|S|} p_{i} \left( \mathbf{1}_{U_{\aseq_{i}}^{(1)}(k)=\ell} + \mathbf{1}_{U_{\aseq_{i}}^{(2)}(k)=\ell}+\cdots + \mathbf{1}_{U_{\aseq_{i}}^{(d)}(k)=\ell}\right).
\end{equation}
Denote its largest eigenvalue $\rho=\rho(M)$ then the process $\xi^{k}$ survives with positive probability if and only if $\rho>1$.
Together with Theorem \ref{theorem:bal rc dary} this gives a useful criterion for recurrence and transience. We now illustrate how to utilize the theorem by giving a few examples.
We note that Theorem \ref{thm:iidrotationT2} covers parts of the examples below.

\begin{example}[rotor-router walk, $\period{0,1,2}$]\label{ex:standardRRW}
In the model given in Example \ref{ex:standrad rr binary tree1} we choose each configuration with probability $1/3$.
We start the rotor walk with $k=N=1$ particle or equivalently we start the ($1$-type) Galton-Watson process $\xi^{1}$ with one particle of type $1$. The mean number of offspring of the Galton-Watson process is
\begin{equation*}
m=\frac13 (0+0) + \frac13 (1+1) +\frac13 (0+1)=1.
\end{equation*}
Hence, the process $\xi^{1}$ a.s.\ dies out and the rotor walk is a.s.\ recurrent.
Now, more generally we choose $\aseq^{(1)}=\period{0,1,2}$ with probability $p_1$, $\aseq^{(2)}=\period{1,2,0}$ with probability $p_2$, and $\aseq^{(3)}=\period{2,0,1}$ with probability $p_3$.
For mean number of offspring we obtain $2p_2+p_3$. Hence, the rotor walk is recurrent if and only if $2p_2+p_3\leq 1$.
\end{example}
\begin{example}[$\period{0,1,0,1,2,2}$]\label{010122}
We analyze further Example \ref{ex:second on binary1} where we choose each configuration with equal probability. We start the rotor with $k=N=2$ particles or equivalently we start the $2$-type Galton-Watson process $\xi^{2}$ with one particle of type $2$. We obtain the following first moment matrix:
\begin{equation*}
M=
\left(
\begin{array}{cc}
 1/3 (1+0+0)  & 1/3(0+2+0)      \\
  1/3 (1+0+0)  & 1/3(0+2+1)
\end{array}
\right)=
\left(
\begin{array}{cc}
 1/3   & 2/3     \\
  1/3   & 1
\end{array}
\right).
\end{equation*}
As the largest eigenvalue of $M$ is $\frac13(2+\sqrt{3})>1$ the rotor walk is a.s.\ transient.
We shall see in the next chapter that Example \ref{010122} also serves as an example for a more specific criterion for transience vs.\ recurrence for uniform rotations.
\end{example}
To end the chapter we analyze the recurrence vs.\ transience regimes when we change the probabilities in  Example \ref{010122}.

\begin{example}\label{010122gen}
Choose $\aseq^{(1)}=\period{0,1,0,1,2,2}$ with probability $p_1$, $\aseq^{(2)}=\period{1,2,1,2,0,0}$ with probability $p_2$,
and $\aseq^{(3)}=\period{2,0,2,0,1,1}$ with probability $p_3$. We obtain the first moment matrix
\begin{equation*}
M=
\left(
\begin{array}{cc}
 p_3 & 2p_2     \\
 p_1 & 2p_2+p_3
\end{array}
\right)
\end{equation*} with largest eigenvalue equal to $p_2+p_3+\sqrt{p_2 (2 p_1+p_2)}$. 
\end{example}

\subsection{Criterion for uniform rotations on the binary tree}
Let $\aseq$ be a rotor sequence.
We call a finite rotor sequence a \emph{piece}. We say that
\begin{equation*}
\alpha=(\underbrace{0,\ldots,0}_{m},\underbrace{1,\ldots,1}_{m},\underbrace{2,\ldots,2}_{m})
\end{equation*} and each of its rotations $\pi^i \alpha$, $i\in\{0,1,2\}$, are \emph{$m$-standard pieces}. A piece is called \emph{standard} if it is $m$-standard for some $m$.
A sequence $\aseq$ is the \emph{concatenation} of standard pieces if $\aseq=(\alpha_{1},\alpha_{2},\ldots)$ for some standard pieces $\alpha_{1},\alpha_{2}\ldots$.
\begin{theorem}[Criterion for uniform rotation]\label{thm:iidrotationT2} Fix a rotor sequence $\aseq$ on $\T_{2}$.
The rotor walk in the uniform rotation model corresponding to $\aseq$ is recurrent a.s.\ if and only if $\aseq$ is a concatenation of standard pieces.
\end{theorem}
\begin{proof}
We shall first check that for any $m$-standard piece $\alpha$ the first moment matrix $M_{\alpha}$ corresponding to $\period{\alpha}$ has spectral radius $1$.
The case $m=1$ gives rise to an $1\times 1$ matrix with entry $1$, see Example \ref{ex:standardRRW}. In the case $m>1$ one checks that the first moment matrix is the  $m\times m $ matrix where all
columns are the $0$-vector except for the $m$th column which is the  $1$-vector. Therefore its eigenvalues are $0$ and $1$.

The general case where $\aseq=(\alpha_{1},\alpha_{2},\ldots,\ldots)$ is a concatenation of standard pieces $\alpha_{i}$ is in the same spirit. However if $\aseq$ is not periodic the multi-type Galton-Watson process $\xi$
may have infinitely many types and the first moment matrix becomes a non-negative operator $M=(m(i,j)_{i,j\in \N}$.
Let $|\alpha_i|$ be the length of $\alpha_{i}$.
 For $k\in \N$ we define
 \begin{equation*}
J(k)=\inf\{j\in \N: \sum_{i=1}^{j} |\alpha_i| \geq k\} \mbox{ and } e(k)= \sum_{i=1}^{J(k)} |\alpha_i|.
\end{equation*}
 Note that $e(k)\ge k$ by definition.
Since all $\alpha_{i}, i\in \N,$ are balanced we have that
\begin{equation}\label{eq:Ubounds}
U_{\aseq}^{(i)}(x)\leq \frac{e(k)}3\mbox{ for all }x\leq \frac{e(k)}3,  i\in\{1,\ldots,d\}.
\end{equation}

Fix $k\in \N$. In order to determine whether the first $k$ excursions are finite a.s.\ it suffices to consider the first moment matrix $M_{k}=(m(i,j))_{i,j\leq e(k)}$ of the multi-type Galton-Watson process with $e(k)/3$ types.
Moreover, Equation (\ref{eq:Ubounds}) implies that $M_{k}$ is a block diagonal matrix consisting of $J(k)$ blocks
$M_{\alpha_{i}}, 1\leq i\leq J(k)$, with additional entries in the lower diagonal part. Since $\rho(M_{\alpha_{i}})=1$ for all $i$ we have that $\rho(M_{k})=1$ and so the first $k$ excursions are finite a.s. Since $k$ was arbitrary, recurrence follows.

Let us now prove that every sequence $\aseq$ that is not a concatenation of standard pieces gives rise to a transient rotor walk. We decompose $\aseq=(\beta,\bseq)$ where $\beta$ is a piece which is a concatenation of standard pieces and
$\bseq$ is a rotor sequence which has, without loss of generality, the form
\begin{equation*}
\bseq=(\underbrace{0,\ldots,0}_{r},\underbrace{1,\ldots,1}_{s},\underbrace{2,\ldots,2}_{t},i,\ldots),\quad i\neq 2.
\end{equation*}
In other words, we have that $s$ and $t$ are not both equal to $r$. We treat here the case where $r<s$. The remaining cases ($r=s\neq t$ and $r>s$) are done analogously. We define $\bseq^{(1)}=\bseq, \bseq^{(2)}=\tau^{1}\bseq, \bseq^{(3)}= \tau^{2}\bseq$, and
\begin{align*}
\cseq_{1}&=(\underbrace{0,\ldots,0}_{r},\underbrace{1,\ldots,1}_{s},0,\ldots),\cr
\cseq_{2}&=(\underbrace{1,\ldots,1}_{r},\underbrace{2,\ldots,2}_{s},0,\ldots),\cr
\cseq_{3}&=(\underbrace{2,\ldots,2}_{r},\underbrace{0,\ldots,0}_{s},0,\ldots).
\end{align*}
That is, $\cseq_{1}, \cseq_{2}, \cseq_{3}$ are the same as $\bseq^{(1)}, \bseq^{(2)},\bseq^{(3)}$ respectively in the indices $1$ to $r+s$ and are $0$ in the larger indices.
The first moment matrix $M_{\cseq}$ of the uniform rotor configuration consisting of $\cseq_{1}, \cseq_{2}$ and $\cseq_{3}$ (that means that each of them chosen with probability $1/3$) is given by
\newcommand\coolleftbrace[2]{%
#1\left\{\vphantom{\begin{matrix}%
#2 \end{matrix}}\right.}
\newcommand\coolrightbrace[2]{%
\left.\vphantom{\begin{matrix}%
#1 \end{matrix}}\right\}#2}
\setcounter{MaxMatrixCols}{11}
\newcommand\coolover[2]{\mathrlap{\smash{%
\overbrace{\phantom{\begin{matrix} #2 %
\end{matrix}}}^{\mbox{$#1$}}}}#2}
\newcommand\coolunder[2]{\mathrlap{\smash{%
\underbrace{\phantom{\begin{matrix} #2 %
\end{matrix}}}_{\mbox{$#1$}}}}#2}
\[
\vphantom{\begin{matrix} &\\ &\\&\\&\\&\\&\\&\\&\\&\\&\end{matrix} }
\frac13
\begin{pmatrix}
\coolover{r}{0 & \cdots & 0 & 2} &0 & \cdots&  0 & 1 &0 &\cdots &0  \cr
\vdots & \ddots  & \vdots &\vdots & \vdots& \ddots&\vdots &\vdots & \vdots  &\ddots & \vdots \cr%
0 & \cdots & 0 & 2 &0 & \cdots &0 & 1 &0 &\cdots &0\cr
0 & \cdots & 0 & 2 &0 & \cdots&  0 &2 &0 &\cdots &0  \cr
\vdots & \ddots  & \vdots &\vdots & \vdots& \ddots&\vdots &\vdots & \vdots  &\ddots & \vdots \cr%
\coolunder{s}{0 & \cdots & 0 & 2 &0 & \cdots &0 & 2} &0 &\cdots &0\cr
\end{pmatrix}%
\begin{matrix}
\coolrightbrace{0 \\ \vdots \\ 0}{r}\\
\coolrightbrace{0 \\ \vdots \\ 0}{s}
\end{matrix}
\]
which  has spectral radius
strictly larger than $1$. Denote by $q$ the length of $\beta$.
The original first moment matrix $M$ restricted to the first $q+r+s+t$-block is minorated by a block diagonal matrix consisting of a first block of size $q$ and a second block which is $M_{\cseq}$.
The fact that $\rho(M_{\cseq})>1$ implies that $\rho(M_\bseq)>1$, hence also $\rho(M)>1$ and the rotor walk is transient a.s.
\end{proof}

\subsection{Transience for uniform rotations on  \texorpdfstring{$\T_{d}, d\geq 3$}{the regular tree}}
The following theorem states that in the uniform rotation model on the $d$-ary tree, $d\ge 3$, the walk is always transient.

\begin{theorem}\label{thm:iidrotationTd}
Let $\aseq$ be a rotor sequence on $\T_{d}, d\geq 3$. Then, the rotor walk in the uniform rotation corresponding to $\aseq$ is transient a.s.\
\end{theorem}
\begin{proof}
Without loss of generality we assume that $\aseq$ starts with $0$.
Let $\pi\in S_{d+1}$ be the rotation mapping $n\mapsto n+1 \mod (d+1)$.
Let $m\ge 1$ be such that $\aseq(1)=\aseq(2)=\ldots=\aseq(m)=0$ and $x:=\aseq(m+1)\ne 0$. For  $j=d-x$ we have $\pi^j (x)=0$.
The set of rotor sequences $\{\pi^i \aseq: i\notin \{ 0,j\}\}$ has $d-1\ge 2$ elements. We shall now show that the first excursion
is infinite with positive probability. To do this we  prove the that branching process  starting with one particle of type  $1$, i.e.\ $k=1$, will survive with positive probability.
We consider first particles of  type  $1$. We have that
\begin{align*}
U_{\aseq}^{(i)}(1)& =0 \quad \forall i\in\{1,\ldots,d\}, \cr
U_{\pi^{j}\aseq}^{(i)}(1)& = m\delta_{ij} \quad \forall i\in\{1,\ldots,d\}, \cr
U_{\pi^{\ell}\aseq}^{(i)}(1)&\geq  m \delta_{i\ell}  +\delta_{i(\ell+x)}  \quad \forall i\in\{1,\ldots,d\}  \forall \ell\notin\{0,j\}.
\end{align*}
Now, by monotonicity of the $U$-functions we also have that for all $k\geq 2$
\begin{align*}
U_{\pi^{j}\aseq}^{(i)}(k)& \geq m\delta_{ij} \quad \forall i\in\{1,\ldots,d\}, \cr
U_{\pi^{\ell}\aseq}^{(i)}(k)&\geq  m \delta_{i\ell}  +\delta_{i(\ell+x)}  \quad \forall i\in\{1,\ldots,d\}  \forall \ell\notin\{0,j\}.
\end{align*}
Hence, the expected number of children  of a particle   is $d/(d+1) + 2(d-1)/(d+1)=(3d-2)/(d+1)>1,$ since $d\geq 3$. 
Denote by $\zeta_{n}=\sum_{i} \xi_{n}(i)$ the number of particles at time $n$ in the original multi-type Galton-Watson process. A standard coupling argument gives that $(\zeta_{n})_{n\geq 1}$ can be stochastically  bounded below by a Galton-Watson process with offspring distribution $q(0)=q(1)=1/(d+1), q(2)=(d-1)/(d+1).$ As the latter survives with positive probability, so does the multi-type Galton-Watson process.
 \end{proof}

\subsection{Conjectures for uniform shifts on the \texorpdfstring{$d$-ary tree, $d\geq 2$}{regular tree}}


Let $\aseq$ be an $L$-periodic rotor sequence on $\T_{2}$.
Denote by $\mathcal{A}_{rec}^{conj}$ the set consisting of sequences of the form $$\period{0,i_{1}, j_{1},0, i_{2}, j_{2}, \ldots, 0, i_{N}, j_{N}} \mbox{ with }\{i_{\ell},j_{\ell}\}=\{1,2\}, 1\leq \ell \leq N,$$ and all its possible shifts.

Similar to the first part of the  proof of Theorem \ref{thm:iidrotationT2} one checks
 that the rotor walk in the uniform shift model corresponding to $\aseq$ is recurrent a.s.\ if  $\aseq\in \mathcal{A}_{rec}^{conj}$.

\begin{conjecture}\label{conj:shiftbinarytree}
Let $\aseq$ be an $L$-periodic rotor sequence on $\T_{2}$. The rotor walk in the uniform shift model corresponding to $\aseq$ is recurrent a.s.\ if and only if $\aseq\in \mathcal{A}_{rec}^{conj}$.
\end{conjecture}
\begin{remark}
The above conjecture holds  for $L\leq 12$. Indeed, we calculated (using a computer) the largest eigenvalues of the first moment matrices for all possible sequences.
\end{remark}

The notion of being ${A}_{rec}^{conj}$ can naturally be generalized to $d$-ary trees.
\begin{lemma}
Let $\aseq$ be a non-degenerate $L$-periodic rotor sequence on $\T_{d}, d\geq 3$.
The rotor walk in the uniform shift model corresponding to  $\aseq$ is transient a.s.\ if $\aseq\in \mathcal{A}_{rec}^{conj}$.
\end{lemma}
\begin{proof}
Let $\aseq$ be such that $\aseq$ and all its shifts are concatenations of  $1$-standard pieces and let $N=L/(d+1)$. Without loss of generality we assume that $\aseq$ starts with $0$. Then $\aseq$ is of the form
$$\aseq=\period{0,i_{1,1}, i_{1,2},\ldots, i_{1,d} ,0, i_{2,1}, i_{2,2},\ldots, i_{2,d}, \ldots, 0,i_{N,1}, i_{N,2},\ldots, i_{N,d}} $$ with $\{i_{i,1},\ldots, i_{i,d}\}=\{1,2,\ldots,d\}$ for all $i\in \{1,2,\ldots,N\}$.
Now, $U_{2+k(d+1)}^{(i)}(1)\geq 1$ for all $i\in\{1,\ldots,d\}$ and $k\in\{0,\ldots, N-1\}$. Moreover, $U_{3+k(d+1)}^{(i)}(1)\geq 1$ for all $i \in \{1,\ldots,d\}\setminus\{i_{1,1}\}$ and $k\in\{0,\ldots, N-1\}$.
Continuing this observation along all shifts leads that the number of $U$-functions verifying $U(1)\geq 1$ is equal to $ N d (d+1)$. Since $U(1)\geq 1$ implies that $U(x)\geq 1$ for all $x\leq N$ the mean number of offspring (of any type)
of a particle (of any type) in the multi-type Galton-Watson process is at least  $N d (d+1)/L=d/2>1$. Hence, the branching process survives with positive probability and the rotor walk is transient a.s.
\end{proof}

\begin{conjecture}\label{conj:shiftdarytree}
For any $L$-periodic rotor sequence in $\T_{d}, d\geq 3$, the rotor walk in the corresponding uniform shift model is transient a.s.
\end{conjecture}



\section{Acknowledgments}
T.O.\ would like to thank Igor Shinkar for stimulating discussions at early stages of the project.
The project was progressed while both authors were visiting the ``Institut f\"ur Mathematische Strukturtheorie'' at TU-Graz. 
The authors are grateful to the institute for its financial support, and to all its members for their warm hospitality.
The authors wish to thank the referees for  their valuable comments that improved the presentation and clarity of the paper. In particular, they are grateful to one of the referees for spotting a gap 
in the proof of Lemma 2.3 and for suggesting one-line proofs for Lemmas 2.11 and 2.12 that were adopted in the current version.   

\begin{thebibliography}{10}

\bibitem{ABO}
Gideon Amir, Noam Berger, and Tal Orenshtein.
\newblock Zero-one law for directional transience of one dimensional excited
  random walks.
\newblock {\em Ann. Inst. Henri Poincar{\'e} Probab. Stat.}, 52(1):47--57,
  2016.

\bibitem{amir2016excited}
Gideon Amir and Tal Orenshtein.
\newblock Excited mob.
\newblock {\em Stochastic Processes and their Applications}, 126(2):439--469,
  2016.

\bibitem{AnHo:11}
Omer Angel and Alexander~E. Holroyd.
\newblock Rotor walks on general trees.
\newblock {\em SIAM J. Discrete Math.}, 25(1):423--446, 2011.

\bibitem{AnHo:12}
Omer Angel and Alexander~E. Holroyd.
\newblock Recurrent rotor-router configurations.
\newblock {\em J. Comb.}, 3(2):185--194, 2012.

\bibitem{AN:72}
Krishna~B. Athreya and Peter~E. Ney.
\newblock {\em Branching processes}.
\newblock Springer-Verlag, New York-Heidelberg, 1972.
\newblock Die Grundlehren der mathematischen Wissenschaften, Band 196.

\bibitem{basdevant2008speed}
Anne-Laure Basdevant and Arvind Singh.
\newblock On the speed of a cookie random walk.
\newblock {\em Probability Theory and Related Fields}, 141(3-4):625--645, 2008.

\bibitem{BS2009recurrence}
Anne-Laure Basdevant and Arvind Singh.
\newblock Recurrence and transience of a multi-excited random walk on a regular
  tree.
\newblock {\em Electron. J. Probab}, 14(55):1628--1669, 2009.

\bibitem{benjamini2003excited}
Itai Benjamini and David~B. Wilson.
\newblock {Excited random walk}.
\newblock {\em Electron. Comm. Probab}, 8(9):86--92, 2003.

\bibitem{CoSp:06}
Joshua~N. Cooper and Joel Spencer.
\newblock Simulating a random walk with constant error.
\newblock {\em Combin. Probab. Comput.}, 15(6):815--822, 2006.

\bibitem{dvoretzky1947problem}
Aryeh Dvoretzky and Theodore Motzkin.
\newblock A problem of arrangements.
\newblock {\em Duke Math. J}, 14(305-313):136, 1947.

\bibitem{FletAl:14}
Laura Florescu, Shirshendu Ganguly, Lionel Levine, and Yuval Peres.
\newblock Escape {R}ates for {R}otor {W}alks in {$\mathbb{Z}^d$}.
\newblock {\em SIAM J. Discrete Math.}, 28(1):323--334, 2014.

\bibitem{FrLe:11}
Tobias Friedrich and Lionel Levine.
\newblock Fast simulation of large-scale growth models.
\newblock In {\em APPROX-RANDOM}, pages 555--566, 2011.

\bibitem{Ha:63}
Theodore~E. Harris.
\newblock {\em The theory of branching processes}.
\newblock Die Grundlehren der Mathematischen Wissenschaften, Bd. 119.
  Springer-Verlag, Berlin; Prentice-Hall, Inc., Englewood Cliffs, N.J., 1963.

\bibitem{harris1952first}
Theodore~Edward Harris.
\newblock First passage and recurrence distributions.
\newblock {\em Transactions of the American Mathematical Society},
  73(3):471--486, 1952.

\bibitem{HoLietal:08}
Alexander~E. Holroyd, Lionel Levine, Karola M{\'e}sz{\'a}ros, Yuval Peres,
  James Propp, and David~B. Wilson.
\newblock Chip-firing and rotor-routing on directed graphs.
\newblock In {\em In and out of equilibrium. 2}, volume~60 of {\em Progr.
  Probab.}, pages 331--364. Birkh{\"a}user, Basel, 2008.

\bibitem{HoPr:10}
Alexander.~E. Holroyd and James Propp.
\newblock {R}otor {W}alks and {M}arkov {C}hains.
\newblock In M.~Marni M.~E.~Lladser, Robert S.~Maier and A.~Rechnitzer,
  editors, {\em Algorithmic Probability and Combinatorics}, volume 520 of {\em
  Contemporary Mathematics}, pages 105--126, 2010.

\bibitem{HuMuSa:15}
Wilfried Huss, Sebastian M{{\"u}}ller, and Ecaterina Sava-Huss.
\newblock Rotor-routing on {G}alton-{W}atson trees.
\newblock {\em Electron. Commun. Probab.}, 20:no. 49, 12, 2015.

\bibitem{HuSa:12}
Wilfried Huss and Ecaterina Sava.
\newblock Transience and recurrence of rotor-router walks on directed covers of
  graphs.
\newblock {\em Electron. Commun. Probab.}, 17:no. 41, 1--13, 2012.

\bibitem{kesten1975limit}
Harry Kesten, Mykyta~V Kozlov, and Frank Spitzer.
\newblock A limit law for random walk in a random environment.
\newblock {\em Compositio Mathematica}, 30(2):145--168, 1975.

\bibitem{Kl:05}
Michael Kleber.
\newblock Goldbug variations.
\newblock {\em Math. Intelligencer}, 27(1):55--63, 2005.

\bibitem{kosygina2008positively}
Elena Kosygina and Martin~PW Zerner.
\newblock Positively and negatively excited random walks on integers, with
  branching processes.
\newblock {\em Electron. J. Probab}, 13(64):1952--1979, 2008.

\bibitem{kosygina2012excited}
Elena Kosygina and Martin~PW Zerner.
\newblock Excited random walks: results, methods, open problems.
\newblock {\em Bull. Inst. Math. Acad. Sin. (N.S.)}, 8:105--157, 2013.

\bibitem{LaLe:09}
Itamar Landau and Lionel Levine.
\newblock The rotor-router model on regular trees.
\newblock {\em J. Combin. Theory Ser. A}, 116(2):421--433, 2009.

\bibitem{Pa:07}
A.~E. Patrick.
\newblock Euler walk on a {C}ayley tree.
\newblock {\em J. Stat. Phys.}, 127(3):629--653, 2007.

\bibitem{PretAl:96}
V.~B. Priezzhev, Deepak Dhar, Abhishek Dhar, and Supriya Krishnamurthy.
\newblock Eulerian walkers as a model of self-organized criticality.
\newblock {\em Phys. Rev. Lett.}, 77(25):5079--5082, Dec 1996.

\bibitem{Sen:06}
Eugene Seneta.
\newblock {\em Non-negative matrices and {M}arkov chains}.
\newblock Springer Series in Statistics. Springer, New York, 2006.
\newblock Revised reprint of the second (1981) edition [Springer-Verlag, New
  York; MR0719544].

\bibitem{Wi:96}
David~B. Wilson.
\newblock Generating random spanning trees more quickly than the cover time.
\newblock In {\em Proceedings of the {T}wenty-eighth {A}nnual {ACM} {S}ymposium
  on the {T}heory of {C}omputing ({P}hiladelphia, {PA}, 1996)}, pages 296--303.
  ACM, New York, 1996.

\end{thebibliography}

\end{document}